\documentclass[11pt,a4paper]{amsart}
\usepackage[utf8]{inputenc}
\usepackage{amsmath, amssymb, amsthm}
\usepackage{geometry}
\usepackage{hyperref}
\usepackage{mathrsfs}

\newtheorem{theorem}{Theorem}[section]

\newtheorem{corollary}[theorem]{Corollary}

\theoremstyle{remark}
\newtheorem{remark}[theorem]{Remark}

\newcommand{\Rd}{\mathbb{R}^d}
\newcommand{\mx}{x}

\title{Sharp Volume-Based Bounds for BMO Functions on Compact Domains and Applications}

\author{M.~Mi\v{s}ur}
\address{Marin Mi\v{s}ur,
University of Zagreb, Faculty of Science, Bijeni\v{c}ka cesta 30,
10000 Zagreb, Croatia.
E-mail: mmisur@math.hr}

\date{\today}



\begin{document}

\begin{abstract}
We establish sharp localized $L^r$ bounds for compactly supported functions of bounded mean oscillation (BMO).
Classical estimates obtained via the John-Nirenberg inequality bound the $L^r$ norm of such a function in terms
of the diameter of its support. We show that the diameter dependence can be replaced by the Lebesgue measure of
the support, the relevant quantity for thin, highly eccentric domains. The proof relies on non-increasing rearrangements
and the Bennett-DeVore-Sharpley theorem, and requires no vanishing-oscillation hypothesis. We present physical and
operator-theoretic applications: geometric collapse in tubular neighborhoods, a localized Cwikel-Lieb-Rozenblum
bound for Schrödinger operators, and decoupling bounds for the localized action of Calderón-Zygmund singular integrals.
\end{abstract}

\maketitle

\section{Introduction}

The space of functions of bounded mean oscillation ($\mathrm{BMO}$) and its closure, the space of vanishing mean
oscillation ($\mathrm{VMO}$), frequently serve as borderline spaces in harmonic analysis and partial differential
equations where $L^\infty$ bounds fail. A common requirement in operator continuity and localized energy bounds is
estimating the $L^r$ norm of a $\mathrm{VMO}(\Rd)$ function restricted to a compact support $K$. 

Previous bounds derived directly from the John-Nirenberg inequality applied on a bounding ball $B \supset K$ yield estimates of the form
\begin{equation}
    \|\varphi\|_{\mathrm{L}^r(K)} \leq C_d\, r\, (\operatorname{diam}K)^{d/r}\, \|\varphi\|_{\mathrm{BMO}(\Rd)}, \quad r \in (1, \infty).
\end{equation}
As noted in prior literature, the appearance of the diameter rather than the Lebesgue measure $|K|$ is a structural
artifact of the enclosing-ball argument. For thin sets (such as tubular neighborhoods), the gap between $|K|^{1/r}$
and $(\operatorname{diam}K)^{d/r}$ becomes unbounded, posing a significant defect in applications where sets degenerate
geometrically. It was recently conjectured in \cite{Misur26} that the estimate holds with $|K|$ in place of $(\operatorname{diam}K)^d$.

Using the Bennett-DeVore-Sharpley theorem \cite{BDS81} on the non-increasing rearrangement of BMO functions,
we prove this, replacing the diameter by the Lebesgue measure of the support. Although VMO supplies the physical
and operator-theoretic motivation, the bound itself requires only membership in BMO together with compact support;
we record this generality in Remark \ref{rem:vmo}.

\section{Main Result and Proof}

\begin{theorem}\label{thm:sharp_volume}
Let $r\in(1,\infty)$ and let $\varphi\in\mathrm{BMO}(\Rd)$ be supported in a compact set $K$. Then
\begin{equation}\label{eq:volume_bound}
\|\varphi\|_{\mathrm{L}^r(K)} \leq C_d\, r\, |K|^{1/r}\, \|\varphi\|_{\mathrm{BMO}(\Rd)},
\end{equation}
where $C_d$ depends only on the dimension.
\end{theorem}

\begin{proof}
Let $\varphi^*$ denote the non-increasing rearrangement of $|\varphi|$, and define the maximal average function $\varphi^{**}$ by
\begin{equation*}
\varphi^{**}(t) = \frac{1}{t}\int_0^t \varphi^*(s)\,ds, \quad t > 0.
\end{equation*}
Recall that the sharp maximal function is defined as $\varphi^\#(\mx) = \sup_{Q \ni \mx} \frac{1}{|Q|} \int_Q |\varphi - \varphi_Q|$.
Because $\varphi \in \mathrm{BMO}(\Rd)$, we have the uniform bound $\varphi^\#(\mx) \leq \|\varphi\|_{\mathrm{BMO}(\Rd)}$
for all $\mx \in \Rd$. Consequently, both its non-increasing rearrangement and the associated maximal average satisfy
$(\varphi^\#)^*(t) \leq (\varphi^\#)^{**}(t) \leq \|\varphi\|_{\mathrm{BMO}(\Rd)}$ for all $t > 0$.
The Bennett-DeVore-Sharpley theorem \cite{BDS81}, valid for locally integrable functions on $\Rd$, states that
$(\varphi^{**} - \varphi^*)(t) \leq C_d (\varphi^\#)^{**}(t)$. Applying this on $\Rd$ and using the bound above,
we control the difference between the maximal average and the rearrangement everywhere on $(0, \infty)$:
\begin{equation}\label{eq:bds}
\varphi^{**}(t) - \varphi^*(t) \leq C_d\, \|\varphi\|_{\mathrm{BMO}(\Rd)},
\end{equation}
where $C_d$ is a purely dimensional constant. 

By definition, the non-increasing rearrangement is given by the infimum $\varphi^*(t) = \inf\{s > 0 : |\{\mx \in \Rd : |\varphi(\mx)| > s\}| \leq t\}$. Since $\varphi$ is supported entirely within the compact set $K$, the measure of its strict support is bounded by the volume of $K$, meaning $|\{\mx \in \Rd : |\varphi(\mx)| > 0\}| \leq |K|$. Therefore, for any $t \geq |K|$ and any $s > 0$, the condition $|\{\mx \in \Rd : |\varphi(\mx)| > s\}| \leq t$ is inherently satisfied. Taking the infimum over all such $s > 0$ yields exactly $\varphi^*(t) = 0$ for all $t \geq |K|$. Evaluating \eqref{eq:bds} exactly at the boundary $t = |K|$ forces the global condition:
\begin{equation}\label{eq:boundary}
\varphi^{**}(|K|) = \varphi^{**}(|K|) - \varphi^*(|K|) \leq C_d\, \|\varphi\|_{\mathrm{BMO}(\Rd)}.
\end{equation}

Furthermore, since $\varphi$ belongs to $\mathrm{BMO}(\Rd)$, the John-Nirenberg inequality \cite{JN61} gives
$\varphi \in L^p_{\mathrm{loc}}(\Rd)$ for every $p < \infty$; combined with the compact support of $\varphi$,
this yields $\varphi \in L^1(\Rd)$. Hence $\varphi^* \in L^1(0,\infty)$, so $s \mapsto \int_0^s \varphi^*$ is
absolutely continuous, and multiplying by the smooth factor $1/t$ shows that $\varphi^{**}$ is absolutely
continuous on every interval $[t_0, |K|]$ with $t_0 > 0$. In particular $\varphi^{**}$ is differentiable almost
everywhere and the fundamental theorem of calculus applies on such intervals. Differentiating $\varphi^{**}$ gives
\begin{equation*}
\frac{d}{dt}\varphi^{**}(t) = -\frac{1}{t^2}\int_0^t \varphi^*(s)\,ds + \frac{1}{t}\varphi^*(t) = -\frac{\varphi^{**}(t) - \varphi^*(t)}{t}.
\end{equation*}
Substituting \eqref{eq:bds} into this derivative yields the differential inequality:
\begin{equation*}
-\frac{d}{dt}\varphi^{**}(t) \leq \frac{C_d}{t}\, \|\varphi\|_{\mathrm{BMO}(\Rd)}.
\end{equation*}
Integrating this bound over the interval $[t, |K|]$ (for $0 < t \leq |K|$) gives:
\begin{equation*}
\varphi^{**}(t) - \varphi^{**}(|K|) \leq C_d\, \|\varphi\|_{\mathrm{BMO}(\Rd)} \ln\left(\frac{|K|}{t}\right).
\end{equation*}
Applying the boundary condition \eqref{eq:boundary} and noting that $\varphi^*(t) \leq \varphi^{**}(t)$,
we establish a logarithmic envelope for the rearrangement:
\begin{equation*}
\varphi^*(t) \leq \varphi^{**}(t) \leq C_d\, \|\varphi\|_{\mathrm{BMO}(\Rd)} \left(1 + \ln\left(\frac{|K|}{t}\right)\right).
\end{equation*}

The $\mathrm{L}^r(\Rd)$ norm of $\varphi$ is identically the $L^r$ norm of its rearrangement on $(0, |K|)$:
\begin{equation*}
\|\varphi\|_{\mathrm{L}^r(K)} = \left( \int_0^{|K|} (\varphi^*(t))^r \,dt \right)^{1/r} \leq C_d\, \|\varphi\|_{\mathrm{BMO}(\Rd)} \left( \int_0^{|K|} \left(1 + \ln\left(\frac{|K|}{t}\right)\right)^r \,dt \right)^{1/r}.
\end{equation*}
Applying the substitution $u = \ln(|K|/t)$ transforms the integration measure to $dt = -|K|e^{-u} \,du$,
mapping the bounds from $(0, |K|)$ to $(\infty, 0)$:
\begin{equation*}
\left( \int_0^{|K|} \left(1 + \ln\left(\frac{|K|}{t}\right)\right)^r \,dt \right)^{1/r} = |K|^{1/r} \left( \int_0^\infty (1+u)^r e^{-u} \,du \right)^{1/r}.
\end{equation*}
By Minkowski's inequality, which strictly applies here on the valid positive measure space $L^r(\mathbb{R}^+, e^{-u}du)$,
this sum is bounded by the sum of the roots:
\begin{equation*}
|K|^{1/r} \left( \left(\int_0^\infty e^{-u} \,du\right)^{1/r} + \left(\int_0^\infty u^r e^{-u} \,du\right)^{1/r} \right) = |K|^{1/r} \big(1 + (\Gamma(r+1))^{1/r}\big).
\end{equation*}
Stirling's approximation dictates that $(\Gamma(r+1))^{1/r} \approx r/e \leq C\, r$.
Factoring this constant back into the norm inequality completes the proof:
\begin{equation*}
\|\varphi\|_{\mathrm{L}^r(K)} \leq C_d'\, r\, |K|^{1/r}\, \|\varphi\|_{\mathrm{BMO}(\Rd)}.
\end{equation*}
\end{proof}

\begin{remark}\label{rem:vmo}
Since $\mathrm{VMO}(\Rd) \subset \mathrm{BMO}(\Rd)$, Theorem \ref{thm:sharp_volume} applies verbatim to
compactly supported $\mathrm{VMO}$ functions; no vanishing-oscillation hypothesis enters the proof.
This resolves the volume-dependence conjecture raised in \cite{Misur26}. The additional vanishing-oscillation
refinement carried by $\mathrm{VMO}$ plays no role in the estimate: the localized $L^r$ capacity of a
compactly supported function is controlled entirely by the coarser $\mathrm{BMO}$ seminorm.
\end{remark}

\begin{remark}
The constant $C_d$ is purely dimensional. It comes from the Bennett-DeVore-Sharpley theorem, whose proof
rests on the Calder\'on-Zygmund decomposition and the John-Nirenberg inequality; the only quantities
entering are the doubling constant of Lebesgue measure (of order $2^d$) and the differentiation basis,
both fixed by the dimension $d$ alone. In particular, $C_d$ is not the constant of the enclosing-ball
argument and carries no dependence on $\varphi$ or its support. Since the non-increasing rearrangement
$\varphi^*$ is determined solely by the Lebesgue measures $|\{|\varphi| > s\}|$ of the level sets, it
discards all topological and metric information about $K$ (sphere, tubular neighborhood, or anisotropic
thin box alike). Consequently, the only trace of the support surviving in \eqref{eq:volume_bound} is its
total volume $|K|$, entering through the boundary condition \eqref{eq:boundary} and the cutoff $\varphi^*(t) = 0$
for $t \geq |K|$. This independence of $C_d$ from the eccentricity of $K$ is what allows $(\operatorname{diam} K)^{d/r}$
to be replaced by the volume factor $|K|^{1/r}$.
\end{remark}

\section{Applications}

The advantage of Theorem \ref{thm:sharp_volume} over previous bounds is clearest when the support degenerates,
that is, when its volume shrinks toward zero while its diameter stays large.

\subsection{Tubular Neighborhoods and Geometric Collapse}
In geometric analysis, one frequently approximates singular measures by functions supported in small tubular
neighborhoods of submanifolds. The sharp volume bound proves that localized $\mathrm{BMO}$ energy must vanish in the limit.

\begin{corollary}\label{cor:tubular_collapse}
Let $M \subset \Rd$ be a compact $k$-dimensional submanifold with $k < d$, and let
$K_\varepsilon = \{x \in \Rd : \operatorname{dist}(x, M) \leq \varepsilon\}$ denote its closed $\varepsilon$-tubular neighborhood.
Suppose $\{\varphi_\varepsilon\}_{\varepsilon > 0}$ is a family of functions in $\mathrm{BMO}(\Rd)$ such that
$\operatorname{supp}(\varphi_\varepsilon) \subseteq K_\varepsilon$ and $\sup_{\varepsilon > 0} \|\varphi_\varepsilon\|_{\mathrm{BMO}(\Rd)} < \infty$.
Then for any $r \in (1, \infty)$,
\begin{equation*}
\|\varphi_\varepsilon\|_{\mathrm{L}^r(\Rd)} \leq C_{M, d}\, r\, \varepsilon^{(d-k)/r} \sup_{\varepsilon > 0} \|\varphi_\varepsilon\|_{\mathrm{BMO}(\Rd)}
\end{equation*}
where the constant $C_{M, d}$ depends only on the dimension and the geometry of $M$.
Consequently, $\|\varphi_\varepsilon\|_{\mathrm{L}^r(\Rd)} \to 0$ as $\varepsilon \to 0$.
\end{corollary}

\begin{proof}
By Weyl's tube formula \cite{Weyl39}, the Lebesgue measure of the $\varepsilon$-tubular neighborhood satisfies
$|K_\varepsilon| \leq C_M\, \varepsilon^{d-k}$ for all sufficiently small $\varepsilon > 0$.
Because $\varphi_\varepsilon$ is supported entirely within the compact set $K_\varepsilon$, Theorem \ref{thm:sharp_volume} applies directly:
\begin{equation*}
\|\varphi_\varepsilon\|_{\mathrm{L}^r(\Rd)} = \|\varphi_\varepsilon\|_{\mathrm{L}^r(K_\varepsilon)} \leq C_d\, r\, |K_\varepsilon|^{1/r}\, \|\varphi_\varepsilon\|_{\mathrm{BMO}(\Rd)}.
\end{equation*}
Inserting the geometric volume bound yields the rate $\varepsilon^{(d-k)/r}$. We note that because
$\operatorname{diam}(K_\varepsilon) \to \operatorname{diam}(M) > 0$ as $\varepsilon \to 0$, the older
diameter-based bounds fail completely to capture this geometric collapse.
\end{proof}

\subsection{Bound States of Schrödinger Operators}
For a Schrödinger operator $H = -\Delta + V$, the Cwikel-Lieb-Rozenblum (CLR) inequality bounds the number
of negative eigenvalues $N(H)$, the number of bound quantum states, in terms of the $L^{d/2}$ norm of the
potential \cite{Lieb76}. When the potential is confined to a highly eccentric geometry, such as a thin
quantum wire, Theorem \ref{thm:sharp_volume} controls the number of bound states through the volume of
the defect rather than its diameter.

\begin{corollary}\label{cor:schrodinger}
Let $H = -\Delta + V$ be a Schr\"odinger operator on $\Rd$ with dimension $d \geq 3$. Suppose the potential
$V$ is in $\mathrm{BMO}(\Rd)$ and its negative part, $V_-(\mx) = \max(0, -V(\mx))$, is supported within a
compact set $K$. Then the number of bound states satisfies
\begin{equation*}
N(H) \leq \widetilde{C}_d \,|K| \,\|V_-\|_{\mathrm{BMO}(\Rd)}^{d/2}.
\end{equation*}
\end{corollary}

\begin{proof}
Because the map $t \mapsto \max(0, -t)$ is $1$-Lipschitz, left composition with a Lipschitz function preserves
bounded mean oscillation; hence $V \in \mathrm{BMO}(\Rd)$ implies
$V_- \in \mathrm{BMO}(\Rd)$ with $\|V_-\|_{\mathrm{BMO}(\Rd)} \leq 2\|V\|_{\mathrm{BMO}(\Rd)}$.
Since $d \geq 3$, the exponent $r = d/2$ falls strictly into the valid range $(1, \infty)$, so Theorem \ref{thm:sharp_volume}
applied to the compactly supported $V_-$ gives
\begin{equation*}
\|V_-\|_{\mathrm{L}^{d/2}(K)} \leq C_d \left(\frac{d}{2}\right) |K|^{2/d} \,\|V_-\|_{\mathrm{BMO}(\Rd)} < \infty.
\end{equation*}
In particular $V_- \in L^{d/2}(\Rd)$, so the integrability hypothesis of the classical CLR inequality is satisfied, and
\begin{equation*}
N(H) \leq L_d \int_{\Rd} |V_-(\mx)|^{d/2} \,d\mx = L_d \,\|V_-\|_{\mathrm{L}^{d/2}(K)}^{d/2},
\end{equation*}
where $L_d$ is the standard CLR constant. Raising the volume bound to the power of $d/2$ and substituting it back into the CLR bound isolates $|K|$:
\begin{equation*}
N(H) \leq L_d \left( C_d \frac{d}{2} |K|^{2/d} \|V_-\|_{\mathrm{BMO}(\Rd)} \right)^{d/2} = \widetilde{C}_d \,|K| \,\|V_-\|_{\mathrm{BMO}(\Rd)}^{d/2}.
\end{equation*}
If $K$ is a thin wire of length $L$ and small thickness $\varepsilon$, its volume scales as $L\varepsilon^{d-1}$
while its diameter is $\sim L$. The classical bound would predict $N(H) \sim L^d$, incorrectly implying an
infinite string of trapped states irrespective of the defect's thickness. The volume bound instead forces this
count to vanish as $\varepsilon \to 0$.
\end{proof}

\subsection{Localized Action of Singular Integrals}
In the study of partial differential operators containing material defects or wave-guides, it is necessary to
bound the localized interference of Calderón-Zygmund operators \cite{CRW76}. If $T$ is a singular integral operator
and $b$ represents a physical defect localized to a thin fracture $K$, the operator's interference across the defect
is heavily constrained by the fracture's volume.

\begin{corollary}\label{cor:localized_action}
Let $T$ be a Calderón-Zygmund singular integral operator bounded on $L^p(\Rd)$ for $1 < p < \infty$.
Suppose a defect symbol $b \in \mathrm{BMO}(\Rd)$ is supported strictly within a compact set $K \subset \Rd$.
Then for any $s \in (1, \infty)$ and $r \in (1, \infty)$ such that $\frac{1}{r} = \frac{1}{s} + \frac{1}{p}$,
the localized pointwise action of the operator satisfies
\begin{equation*}
\|b \, Tf\|_{\mathrm{L}^r(K)} \le C_{d, p, s} \,s\, |K|^{1/s} \|b\|_{\mathrm{BMO}(\Rd)} \|f\|_{\mathrm{L}^p(\Rd)}
\end{equation*}
for all $f \in L^p(\Rd)$, where the constant $C_{d, p, s}$ depends only on the dimension, the exponents,
and the $L^p$ operator norm of $T$.
\end{corollary}

\begin{proof}
By definition, the operator $T$ maps $f \in L^p(\Rd)$ to $Tf \in L^p(\Rd)$ with a bounded operator norm
$\|Tf\|_{\mathrm{L}^p(\Rd)} \le C_T \|f\|_{\mathrm{L}^p(\Rd)}$. Since $b$ is supported in the compact set $K$,
the pointwise multiplication $b(x)Tf(x)$ naturally vanishes outside of $K$. Consequently, we evaluate the $L^r$
norm strictly over the support region $K$. Applying Hölder's inequality on $K$ with the conjugate exponents $p$ and $s$ yields:
\begin{equation*}
\|b \, Tf\|_{\mathrm{L}^r(K)} \le \|b\|_{\mathrm{L}^s(K)} \|Tf\|_{\mathrm{L}^p(K)} \le \|b\|_{\mathrm{L}^s(K)} \|Tf\|_{\mathrm{L}^p(\Rd)}.
\end{equation*}
Since $s \in (1, \infty)$ and $b$ is a compactly supported BMO function, Theorem \ref{thm:sharp_volume} directly
bounds the $L^s$ norm of the defect symbol:
\begin{equation*}
\|b\|_{\mathrm{L}^s(K)} \le C_d\, s\, |K|^{1/s}\, \|b\|_{\mathrm{BMO}(\Rd)}.
\end{equation*}
Substituting this sharp volume-dependent bound and the $L^p$ boundedness of $T$ back into the Hölder estimate
produces the final inequality:
\begin{equation*}
\|b \, Tf\|_{\mathrm{L}^r(K)} \le C_d\, s\, |K|^{1/s} \|b\|_{\mathrm{BMO}(\Rd)} \big(C_T \|f\|_{\mathrm{L}^p(\Rd)}\big) = C_{d, p, s} \,s\, |K|^{1/s} \|b\|_{\mathrm{BMO}(\Rd)} \|f\|_{\mathrm{L}^p(\Rd)}.
\end{equation*}
\end{proof}

The physical significance of Corollary \ref{cor:localized_action} is clearest for thin material defects.
If the defect $K$ is highly eccentric (for example, a narrowing fracture of length $L$ and width $\varepsilon$),
its volume $|K|$ vanishes as $\varepsilon \to 0$ while its diameter stays macroscopic.
The explicit $|K|^{1/s}$ factor then shows that the localized interference decouples from the surrounding space
as the defect thins, in line with the expected physical behavior.

\section{Future Work}
The volume bound of Theorem \ref{thm:sharp_volume} suggests several further directions, particularly in the analysis
of borderline regularity for fluid flows. In physical systems such as those governed by the Navier-Stokes or Euler
equations, intense vorticity is frequently confined to a boundary layer of microscopic thickness $\delta$ but
macroscopic length $L$. Using the volume bound $|K|^{1/r} = (\delta L^{d-1})^{1/r}$ rather than the macroscopic diameter
of the fluid domain gives a route to showing that the localized kinetic energy contributions of BMO turbulence vanish as
the boundary layer thickness $\delta \to 0$. Establishing these energy estimates and evaluating the necessary Biot-Savart
interactions and spatial cutoff constraints remains a subject for future research.

A second direction concerns the localized compactness of commutators. For $b \in \mathrm{BMO}(\Rd)$ the commutator
$[b, T]f = b\,Tf - T(bf)$ is bounded on $L^p(\Rd)$ \cite{CRW76}, and for $b \in \mathrm{VMO}(\Rd)$ it is compact \cite{Uchiyama78}.
Corollary \ref{cor:localized_action} controls only the multiplication term $b\,Tf$, which inherits the compact support of $b$;
the remaining term $T(bf)$ is delocalized by $T$, so the clean volume factor $|K|^{1/s}$ is not immediately available for the full commutator.
Quantifying the compactness of $[b, T]$ for defects on thin sets, for instance operator-norm decay rates as $|K| \to 0$
obtained by approximating $b$ and separately estimating the delocalized tail $T(bf)$, would be a natural next step.
It is here that the vanishing oscillation of $\mathrm{VMO}$ becomes essential rather than merely motivational.

Finally, the rearrangement argument behind Theorem \ref{thm:sharp_volume} uses only the measure-theoretic envelope provided
by the Bennett-DeVore-Sharpley inequality, with no appeal to the linear or metric structure of $\Rd$ beyond a doubling property.
The bound should therefore extend to spaces of homogeneous type $(X, d, \mu)$ in the sense of Coifman and Weiss \cite{CW71},
with the Lebesgue volume $|K|$ replaced by the doubling measure $\mu(K)$ and the dimensional constant $C_d$ replaced by one
depending only on the doubling constant of $\mu$.

\end{document}